\documentclass[a4paper,twoside]{article}
\usepackage{amssymb}
\usepackage{amsmath}
\usepackage{fancyhdr}
\usepackage{amsthm}
\usepackage{rotate}
\usepackage{graphicx}
\usepackage[T1]{fontenc}
\usepackage{afterpage}

\theoremstyle{plain}
\newtheorem{theorem}{Theorem}[section]
\newtheorem{lemma}[theorem]{Lemma}
\newtheorem{proposition}[theorem]{Proposition}

\newtheorem{corollary}[theorem]{Corollary}
\theoremstyle{definition}

\newtheorem{example}[theorem]{Example}

\begin{document}
\afterpage{\rhead[]{\thepage} \chead[W. A. Dudek]{Parastrophes of quasigroups } \lhead[\thepage]{} }

\begin{center}
\vspace*{2pt} {\Large \textbf{Parastrophes of quasigroups}}\\[30pt]
{\large \textsf{\emph{Wieslaw A. Dudek}}}
\\[30pt]
\end{center}
\noindent{\footnotesize \textbf{Abstract.} Parastrophes (conjugates) of a quasigroup can be divided into separate classes containing isotopic parastrophes. We prove that the number of such classes is always 1, 2, 3 or 6. Next we characterize quasigroups having a fixed number of such classes.} 


\section*{\centerline{1. Introduction}}\setcounter{section}{1}

Denote by $S_Q$ the set of all permutations of the set $Q$. We say that a quasigroup $(Q,\cdot)$ is {\em isotopic} to $(Q,\circ)$ if there are $\alpha,\beta,\gamma\in S_Q$ such that $\alpha(x)\circ\beta(y) = \gamma(x\cdot y)$ for all $x,y\in Q$. The triplet $(\alpha,\beta,\gamma)$ is called an {\em isotopism}. 
Quasigroups $(Q,\cdot)$ and $(Q,\circ)$ for which there are $\alpha,\beta,\gamma\in S_Q$ such that $\alpha(x)\circ\beta(y) = \gamma(y\cdot x)$ for all $x,y\in Q$ are called {\em anti-isotopic}. This fact is denoted by $(Q,\cdot)\sim(Q,\circ)$. In the case when $(Q,\cdot)$ and $(Q,\circ)$ are isotopic we write $(Q,\cdot)\approx (Q,\circ)$. It is clear that the relation $\approx$ is an equivalence and divides all quasigroups into disjoint classes containing isotopic quasigroups.

Each quasigroup $Q=(Q,\cdot)$ determines five new quasigroups $Q_i=(Q,\circ_i)$ with the operations $\circ_i$ defined as follows:
$$
\begin{array}{cccc}
x\circ_1 y=z\longleftrightarrow x\cdot z=y\\
x\circ_2 y=z\longleftrightarrow z\cdot y=x\\
x\circ_3 y=z\longleftrightarrow z\cdot x=y\\
x\circ_4 y=z\longleftrightarrow y\cdot z=x\\
x\circ_5 y=z\longleftrightarrow y\cdot x=z\\
\end{array}
$$
Such defined (not necessarily distinct) quasigroups are called {\em parastrophes} or {\em conjugates} of $Q$. Traditionally they are denoted as

\medskip\centerline{ 
$Q_1=Q^{-1}=(Q,\backslash), \ \ Q_2={}^{-1}\!Q=(Q,/), \ \ Q_3= {}^{-1}\!(Q^{-1})=(Q_1)_2$,}

\medskip\centerline{$Q_4=({}^{-1}\!Q)^{-1}=(Q_2)_1$ \ and \ $Q_5=({}^{-1}\!(Q^{-1}))^{-1}=((Q_1)_2)_1=((Q_2)_1)_2$.}

\medskip
Each parastrophe $Q_i$ can be obtained from $Q$ by the permutation $\sigma_i$, where $\sigma_1=(23)$, $\sigma_2=(13)$, $\sigma_3=(132)$, $\sigma_4=(123)$, $\sigma_5=(12)$.

Generally, parastrophes $Q_i$ do not save properties of $Q$. Parastrophes of a group are not a group, but parastrophes of an idempotent quasigroup also are idempotent quasigroups. Moreover, in some cases (described in \cite{Lind}) parastrophes of a given quasigroup $Q$ are pairwise equal or all are pairwise distinct (see also \cite{BPop} and \cite{Pop}). In \cite{Lind} it is proved that the number of distinct parastrophes of a quasigroup is always a divisor of $6$ and does not depend on the number of elements of a quasigroup.

Parastrophes of each quasigroup can be divided into separate classes containing isotopic parastrophes. We prove that the number of such classes is always $1$, $2$, $3$ or $6$. The number of such classes depends on the existence of an anti-isotopism of a quasigroup and some parastrophe of it.

\section*{\centerline{2. Classification of parastrophes}}\setcounter{section}{2}

As it is known (see for example \cite{Bel}) a quasigroup $(Q,\cdot)$ can be considered as an algebra $(Q,\cdot,\backslash,/)$ with three binary operations satisfying the following axioms
$$
x(x\backslash z)=z, \ \ \ (z/y)y=z, \ \ \ x\backslash xy=y, \ \ \ xy/y=x,
$$
where 
$$
x\backslash z=y\longleftrightarrow xy=z \ \ \ {\rm and} \ \ \ z/y=x\longleftrightarrow xy=z.
$$

We will use these axioms to show the relationship between parastrophes.
But let's start with the following simple observation.

\begin{lemma}\label{L-21}
Let $Q$ be a quasigroup. Then

\smallskip
$(a)$ \ $xy=y\circ_5x$, \ $x\circ_1 y=y\circ_3 x$, \ $x\circ_2 y=y\circ_4 x$,\vspace*{1mm}

$(b)$ \ $Q\sim Q_5$, \ $Q_1\sim Q_3$, \ $Q_2\sim Q_4$,\vspace*{1mm}

$(c)$ \ $xy=yx\longleftrightarrow Q=Q_5\longleftrightarrow Q_1=Q_3\longleftrightarrow Q_2=Q_4$,\vspace*{1mm}

$(d)$ \ $Q_1=Q\longleftrightarrow Q_2=Q_3\longleftrightarrow Q_4=Q_5$,\vspace*{1mm}

$(e)$ \ $Q_2=Q\longleftrightarrow Q_1=Q_4\longleftrightarrow Q_3=Q_5$.
\end{lemma}

To describe the relationship between the parastrophes, we will need these two simple lemmas.

\begin{lemma}\label{L-22}
Let $A,B,C,D$ be quasigroups. Then

\smallskip
$(a)$ \ $A\sim B, \ B\sim C\longrightarrow A\approx C$,\vspace*{1mm}

$(b)$ \ $A\sim B, \ B\approx C\longrightarrow A\sim C$,\vspace*{1mm}

$(c)$ \ $A\approx B, \ B\sim C\longrightarrow A\sim C$.
\end{lemma}


\begin{lemma}\label{L-23}
Let $Q_i^{\circ}$ be the $i$-th parastrophe of the quasigroup $Q^{\circ}=(Q,\circ)$.
Then

\smallskip
$(a)$ \ $Q\approx Q^{\circ}$ implies $Q_i\approx Q_i^{\circ}$ for each $i=1,2,3,4,5$,\vspace*{1mm}

\smallskip
$(b)$ \ $Q_i\approx Q_i^{\circ}$ for some $i=1,2,3,4,5$ implies $Q\approx Q^{\circ}.$ \vspace*{1mm}

\smallskip
$(c)$ \ Moreover, if $Q\approx Q^{\circ}$, then for each $i=1,\ldots,5$\vspace{2mm}

 \centerline{$Q\sim Q_i\longleftrightarrow Q^{\circ}\sim Q_i^{\circ},$ and 
 \ $Q\approx Q_i\longleftrightarrow Q^{\circ}\approx Q_i^{\circ}$.}
\end{lemma}

Now we will present a series of lemmas about anti-isotopies of quasigroups and their parastrophes.

\begin{lemma}\label{L-24}
$Q\sim Q\longleftrightarrow Q\approx Q_5\longleftrightarrow Q_1\approx Q_3\longleftrightarrow Q_2\approx Q_4$.
\end{lemma}
\begin{proof}
Indeed, 

\medskip
\centerline{$
Q\sim Q\longleftrightarrow \gamma(xy)=\alpha(y)\beta(x)\longleftrightarrow\gamma(xy)=\beta(x)\circ_5\alpha(y)\longleftrightarrow Q\approx Q_5.
$}

\medskip\noindent
Also

\medskip\noindent
$\rule{2mm}{0mm}Q\sim Q\longleftrightarrow \gamma(xy)=\alpha(y)\beta(x)\longleftrightarrow \gamma(z)=\alpha(y)\beta(z/y)\longleftrightarrow \alpha(y)\backslash\gamma(z)=\beta(z/y).
$ 

\medskip
Thus $Q\sim Q\longleftrightarrow Q_1\sim Q_2$. Moreover,

\medskip
$\arraycolsep=.5mm\begin{array}{rl}
\rule{10mm}{0mm}Q_1\sim Q_2&\longleftrightarrow \alpha(y)\backslash\gamma(z)=\beta(z/y)\longleftrightarrow
\gamma(z)=\alpha(y)\beta(z/y)\\[3pt]
 &\longleftrightarrow \beta(z/y)=\gamma(z)\circ_4\alpha(y)\longleftrightarrow Q_2\approx Q_4.
\end{array}
$ 

\medskip
Similarly, for some $\alpha',\beta',\gamma'\in S_Q$ we have

\medskip
$\arraycolsep=.5mm\begin{array}{rl}
\rule{10mm}{0mm}Q_1\sim Q_2&\longleftrightarrow \gamma'(x/y)=\alpha'(y)/\beta'(x)\longleftrightarrow\gamma'(x\backslash y)\beta'(x)=\alpha'(y)\\[3pt]
&\longleftrightarrow \gamma'(x\backslash y)=\beta'(x)\circ_3\alpha'(y)\longleftrightarrow Q_1\approx Q_3.
\end{array}$ 

\medskip
This completes the proof.
\end{proof}

\begin{lemma}\label{L-25}
$Q\sim Q_1\longleftrightarrow Q\sim Q_2\longleftrightarrow Q_1\approx Q_2$.
\end{lemma}
\begin{proof} 
Indeed, according to the definition of the operations $\backslash$ and $/$, we have
$$
\gamma(x\backslash z)=\alpha(z)\beta(x)\longleftrightarrow \gamma(y)=\alpha(xy)\beta(x)\longleftrightarrow \alpha(xy)=\gamma(y)/\beta(x).
$$
So, $Q_1\sim Q\longleftrightarrow Q\sim Q_2$, which by Lemma \ref{L-22} implies $Q_1\approx Q_2$.

\smallskip
Conversely, if $Q_1\approx Q_2$, then 
$
\gamma(x\backslash y)=\alpha(x)/\beta(y),
$
i.e., $\gamma(x\backslash y)\beta(y)=\alpha(x)$ for some $\alpha,\beta,\gamma\in S_Q$. 
From this, for $y=xz$, we obtain $\gamma(z)\beta(xz)=\alpha(x)$, i.e., $\beta(xz)=\gamma(z)\backslash\alpha(x)$. Thus, $Q\sim Q_1$, and consequently, also $Q\sim Q_2$.
\end{proof}

\begin{lemma}\label{L-26} For any quasigroup $Q$ 

\smallskip
$(a)$ \ $Q_1\sim Q\longleftrightarrow Q_1\sim Q_3\longleftrightarrow Q\approx Q_3\longleftrightarrow Q_1\approx Q_5$,

\smallskip
$(b)$ \ $Q_2\sim Q\longleftrightarrow Q_2\sim Q_4\longleftrightarrow Q\approx Q_4\longleftrightarrow Q_2\approx Q_5$.
\end{lemma}
\begin{proof}
Replacing in Lemma \ref{L-25} a quasigroup $Q$ by $Q_1$ we get the first two equi\-valences. The third equivalence is a consequence of Lemma \ref{L-23}.

Similarly, replacing $Q$ by $Q_2$ we obtain $(b)$. 
\end{proof}

\begin{lemma}\label{L-27}
$Q_3\sim Q\longleftrightarrow Q\approx Q_2\longleftrightarrow Q_1\approx Q_4 \longleftrightarrow Q_3\approx Q_5$.
\end{lemma}
\begin{proof}
Obviously $Q_3\sim Q\longleftrightarrow Q_3\approx Q_5$. Moreover,
$$
Q_3\sim Q\longleftrightarrow \gamma(xy)\alpha(y)=\beta(x)\longleftrightarrow \gamma(xy)=\beta(x)/\alpha(y)\longleftrightarrow Q\approx Q_2.
$$ 
Analogously, $xy=z$ we obtain 
$$Q_3\sim Q\longleftrightarrow
\gamma(z)\alpha(x\backslash z)=\beta(x)\longleftrightarrow Q_1\approx Q_4.
$$
This completes the proof.
\end{proof}

\begin{lemma}\label{L-28}
$Q_4\sim Q\longleftrightarrow Q\approx Q_1\longleftrightarrow Q_2\approx Q_3\longleftrightarrow Q_4\approx Q_5$.
\end{lemma}
\begin{proof}
Of course $Q_4\sim Q\longleftrightarrow Q_4\approx Q_5$. Since 
$Q_4\sim Q\longleftrightarrow \beta(x)\gamma(xy)=\alpha(y)$, we obtain $Q_4\sim Q\longleftrightarrow Q\approx Q_1$ and $Q_4\sim Q\longleftrightarrow Q_2\approx Q_3$ for $x=z/y$. 
\end{proof}

\begin{theorem}\label{TQ}
All parastrophes of a quasigroup $Q$ are isotopic to $Q$ if and only if $Q\sim Q$ and $Q\sim Q_i$ for some $i=1,2,3,4$.
\end{theorem}
\begin{proof}
If $Q\sim Q$, then, by Lemma \ref{L-24}, we have $Q\approx Q_5$, $Q_1\approx Q_3$ and $Q_2\approx Q_4$.
This for $Q\sim Q_i$, $i=1,2,3,4$, by Lemmas \ref{L-26}, \ref{L-27} and \ref{L-28}, gives $Q\approx Q_1\approx Q_2\approx Q_3\approx Q_4\approx Q_5$. So, in this case all parastrophes are isotopic to $Q$.

The converse statement is obvious.
\end{proof}

\begin{corollary}\label{C210}
If $Q\sim Q$ and $Q\sim Q_i$ for some $i=1,2,3,4$, then also $Q\sim Q_i$ for other $i=1,2,3,4,5$.
\end{corollary}

\begin{theorem}\label{TQ1}
A quasigroup $Q$ has exactly two classes of isotopic parastrophes if and only if 

\smallskip
$(1)$ \ $Q\not\sim Q$, $Q\sim Q_1$ and $Q\not\sim Q_i$ for $i=2,3,4$, or equivalently, \vspace*{1mm}

$(2)$ \ $Q\not\sim Q$, $Q\sim Q_2$ and $Q\not\sim Q_i$ for $i=1,3,4$. \vspace*{1mm}

\noindent
In this case $Q\approx Q_3\approx Q_4$ and $Q_1\approx Q_2\approx Q_5$.
\end{theorem}
\begin{proof}
Let $Q$ have exactly two classes of isotopic parastrophes. Then it must be true that $Q\approx Q_i$ for some $i=1,2,3,4,5$ because $Q\not\approx Q_i$ for all $i=1,2,3,4,5$ gives $Q_1\approx Q_j$ for some $j$ which by previous lemmas implies $Q\approx Q_k$ for some $k$.

\smallskip
{\sc Case} $Q\approx Q_1$. In this case $Q_2\approx Q_3$ and $Q_4\approx Q_5$ (Lemma \ref{L-28}). So, the following classes of isotopic parastrophes are possible: 

\smallskip
$1)$ \ $\{Q,Q_1,Q_2,Q_3\}$, \ $\{Q_4,Q_5\}$, 

\smallskip
$2)$ \ $\{Q,Q_1,Q_4,Q_5\}$, \ $\{Q_2,Q_3\}$, 

\smallskip
$3)$ \ $\{Q,Q_1\}$, \ $\{Q_2,Q_3,Q_4,Q_5\}$.

\smallskip
In the first case from $Q_1\approx Q_3$, by Lemma \ref{L-24}, we conclude $Q\approx Q_5$ which shows that in this case we have only one class. This contradics our assumption on the number of classes. So, this case is impossible.

In the second case, $Q\approx Q_5$, by the same lemma, implies $Q_2\approx Q_4$ which (similarly as in previous case) is impossible. Also the third case is impossible because $Q_2\approx Q_4$ leads to $Q_1\approx Q_3$.  Hence must be $Q\not\approx Q_1$.

\smallskip
{\sc Case} $Q\approx Q_2$. Then, according to Lemma \ref{L-27}, $Q_1\approx Q_4$ and $Q_3\approx Q_5$. Thus

\smallskip
$1)$ \ $\{Q,Q_1,Q_2,Q_4\}$, \ $\{Q_3,Q_5\}$, or

\smallskip
$2)$ \ $\{Q,Q_2,Q_3,Q_5\}$, \ $\{Q_1,Q_4\}$, or

\smallskip
$3)$ \ $\{Q,Q_2\}$, \ $\{Q_1,Q_3,Q_4,Q_5\}$.

\smallskip
Using the same argumentation as in the case $Q\approx Q_1$ we can see that the case $Q\approx Q_2$ is impossible.

\smallskip
{\sc Case} $Q\approx Q_3$. By Lemmas \ref{L-21}, \ref{L-22} and \ref{L-25} only the following classes are possible: $\{Q,Q_3,Q_4\}$ and $\{Q_1,Q_2,Q_5\}$. In this case $Q\not\sim Q$ (Lemma \ref{L-24}) and $Q\sim Q_1$ (Lemma \ref{L-26}). Then also $Q\sim Q_2$ (Lemma \ref{L-25}).

\smallskip
{\sc Case} $Q\approx Q_4$. Analogously as $Q\approx Q_3$.

\smallskip
{\sc Case} $Q\approx Q_5$. Then $Q_1\approx Q_3$ and $Q_2\approx Q_4$. Is a similar way as for $Q\approx Q_1$ we can verify that this case is not possible.

\smallskip
So, if $Q$ has exactly two classes of isotopic parastrophes, then $Q\not\sim Q$ and $Q\sim Q_1$, or $Q\not\sim Q$ and $Q\sim Q_2$.

\smallskip
Conversely, if $Q\not\sim Q$ and $Q\sim Q_1$, or equivalently, $Q\not\sim Q$ and $Q\sim Q_2$, then by Lemmas \ref{L-25} and \ref{L-26} we have two classes: $\{Q,Q_3,Q_4\}$ and $\{Q_1,Q_2,Q_5\}$. Since
$Q_1\not\approx Q_3$ (Lemma \ref{L-24}), these classes are disjoint.
\end{proof}

\begin{theorem}\label{TQ3}
A quasigroup $Q$ has exactly three classes of isotopic parastrophes if and only if 

\smallskip
$(1)$ \ $Q\not\sim Q$, $Q\sim Q_3$ and $Q\not\sim Q_i$ for $i=1,2,4$, or \vspace*{1mm}

$(2)$ \ $Q\not\sim Q$, $Q\sim Q_4$ and $Q\not\sim Q_i$ for $i=1,2,3$, or\vspace*{1mm}

$(3)$ \ $Q\sim Q$, $Q\sim Q_5$ and $Q\not\sim Q_i$ for $i=1,2,3,4$.\vspace*{1mm}

\noindent
In the first case we have $\{Q,Q_2\}$, $\{Q_1,Q_4\}$ and $\{Q_3,Q_5\}$; in the second $\{Q,Q_1\}$, $\{Q_2,Q_3\}$ and $\{Q_4,Q_5\}$; in the third $\{Q,Q_5\}$, $\{Q_1,Q_3\}$ and $\{Q_2,Q_4\}$.
\end{theorem}
\begin{proof}
Suppose that a quasigroup $Q$ has exactly three classes of isotopic parastro\-phes. From the above lemmas it follows that in this case $Q\approx Q_i$ for some $i$.

\smallskip
{\sc Case} $Q\approx Q_1$. Then, by Lemma \ref{L-28}, we have three classes $\{Q,Q_1\}$, $\{Q_2,Q_3\}$, $\{Q_4,Q_5\}$ and $Q\sim Q_4$. Since $Q_1\not\approx Q_3$ we also have $Q\not\sim Q$ (Lemma \ref{L-24}).

\smallskip
{\sc Case} $Q\approx Q_2$. In this case  $\{Q,Q_2\}$, $\{Q_1,Q_4\}$, $\{Q_3,Q_5\}$ and $Q\sim Q_3$ (Lemma \ref{L-27}). Analogously as in the previous case  $Q_1\not\approx Q_3$ gives $Q\not\sim Q$.

\smallskip
{\sc Case} $Q\approx Q_3$. This case is impossible because by Lemmas \ref{L-25} and \ref{L-26} it leads to two classes.

\smallskip
{\sc Case} $Q\approx Q_4$. Analogously as $Q\approx Q_3$.

\smallskip
{\sc Case} $Q\approx Q_5$. Then $Q_1\approx Q_3$, $Q_2\approx Q_4$ and $Q\sim Q$. Since classes 
 $\{Q,Q_5\}$, $\{Q_1,Q_3\}$, $\{Q_2,Q_5\}$ are disjoint $Q\not\sim Q_i$ for each $i=1,2,3,4$.

\smallskip
The converse statement is obvious.
\end{proof}

As a consequence of the above results we obtain

\begin{corollary}\label{C213}
Parastrophes of a quasigroup $Q$ are non-isotopic if and only if $Q\not\sim Q$ and $Q\not\sim Q_i$ for all $i=1,2,3,4$.
\end{corollary}

\begin{corollary}\label{C214}
The number of non-isotopic parastrophes of a quasigroup $Q$ is always $1$, $2$, $3$, or $6$. 
\end{corollary}

Depending on the relationship between parastrophes quasigroups can be divided into six types presented below.

\medskip
$$
\begin{array}{|c|l|ll} \hline
{\rm type}&\ \ \  {\rm classes\; of\; isotoipic\; parastrophes}\rule{0mm}{3.5mm}\\ \hline
A&\{Q,Q_1,Q_2,Q_3,Q_4,Q_5\}\rule{0mm}{3.4mm}\\ \hline
B&\{Q,Q_3,Q_4\}, \ \ \{Q_1,Q_2,Q_5\}\rule{0mm}{3.4mm}\\ \hline
C&\{Q,Q_2\}, \ \ \{Q_1,Q_4\}, \ \ \{Q_3,Q_5\}\rule{0mm}{3.4mm}\\ \hline
D&\{Q,Q_1\}, \ \ \{Q_2,Q_3\}, \ \ \{Q_4,Q_5\}\rule{0mm}{3.4mm}\\ \hline
E&\{Q,Q_5\}, \ \ \{Q_1,Q_3\}, \ \ \{Q_2,Q_4\}\rule{0mm}{3.4mm}\\ \hline
F&\{Q\}, \ \{Q_1\}, \ \{Q_2\}, \ \{Q_3\}, \ \{Q_4\}, \ \{Q_5\}\rule{0mm}{3.4mm}\\ \hline
\end{array}
$$

\medskip
Our results are presented in the following table where "$+$" means that the corresponding relation holds. The symbol "$-$" means that this relation has no place.

\medskip
$$
\begin{array}{|c|c|c|c|c|c|c|c|c|c|}\hline
Q\sim Q  &+&-&-&-&-&+&-&Q\approx Q_5\rule{0mm}{3.4mm}\\ \hline
Q\sim Q_1&+&+&-&-&-&-&-&Q\approx Q_3\rule{0mm}{3.4mm}\\ \hline
Q\sim Q_2&+&-&+&-&-&-&-&Q\approx Q_4\rule{0mm}{3.4mm}\\ \hline
Q\sim Q_3&+&-&-&+&-&-&-&Q\approx Q_2\rule{0mm}{3.4mm}\\ \hline
Q\sim Q_4&+&-&-&-&+&-&-&Q\approx Q_1\rule{0mm}{3.4mm}\\ \hline
{\rm type}&A&B&B&C&D&E&F&\rule{0pt}{10pt}\\ \hline
\end{array}
$$

\medskip
The parastrophe $Q_5$ plays no role in our research since always is $Q\sim Q_5$.

\section*{\centerline{3. Parastrophes of selected quasigroups}}\setcounter{section}{3}\setcounter{theorem}{0}

In this section we present characterizations of parastrophes of several classical types of quasigroups. We start with parastrophes of IP-quasigroups.

As a consequence of our results, we get the following well-known fact (see for example  \cite{Bel})
\begin{proposition}\label{P31}
All parastrophes of an $IP$-quasigroup are isotopic.
\end{proposition}
\begin{proof}
Indeed, in any $IP$-quasigroup $Q$ there are permutations $\alpha,\beta\in S_Q$ such that $\alpha(x)\cdot xy=y=yx\cdot\beta(x)$ for all $x,y\in Q$. So, $Q\approx Q_1\approx Q_2$, i.e., $Q$ is a quasigroup of type $A$.
\end{proof}

\begin{corollary}\label{C32}
Parastrophes of a group are isotopic.
\end{corollary}

The same is true for the parastrophes of Moufang quasigroups since groups and Moufang quasigroups are IP-quasigroups.

Also parastrophes of T-quasigroups, linear and alinear quasigroups (studied in \cite{BT}) are isotopic. 
This fact follows from more general result proved below.

\begin{theorem}\label{T33}
All parastrophes of a quasigroup isotopic to a group are isotopic.
\end{theorem}
\begin{proof}
 Let $G=(G,\circ)$ be a group. Then $\varphi(x\circ y)=\varphi(y)\circ\varphi(x)$ for $\varphi(x)=x^{-1}$. Since $(Q,\cdot)\approx (G,\circ)$, for some $\alpha,\beta,\gamma$ we have
$$
\gamma(xy)=\alpha(x)\circ\beta(y)=\varphi^{-1}(\varphi\beta(y)\circ\varphi\alpha(x))=\varphi^{-1}\gamma\left(\alpha^{-1}\varphi\beta(y)\cdot \beta^{-1}\varphi\alpha(x)\right).
$$
Thus $\,\gamma^{-1}\varphi\gamma(xy)=\alpha^{-1}\varphi\beta(y)\cdot \beta^{-1}\varphi\alpha(x)$. So, $Q\sim Q$. 

Moreover, from $\gamma(xy)=\alpha(x)\circ\beta(y)$ for $xy=z$ we obtain
$$
\alpha(x)\backslash\!\!\backslash \gamma(z)=\beta(x\backslash z) \ \ \ {\rm and} \ \ \ \gamma(z) /\!\!/\beta(y)=\alpha(z/y),
$$
where $\backslash\!\!\backslash$ and $/\!\!/$ are inverse operations in a group $G$. Thus $Q_1\approx G_1$ and $Q_2\approx G_2$. Since $G\approx G_1\approx G_2$, also $Q\approx Q_1\approx Q_2$. This shows that a quasigroup isotopic to a group is a quasigroup of type $A$. Hence (Lemma \ref{L-23}) all its parastrophes are isotopic.
\end{proof}

{\em $D$-loops} (called also {\em loops with anti-automorphic property}) are defined as loops with the property $(xy)^{-1}=y^{-1}x^{-1}$, where $x^{-1}$ denotes the inverse element \cite{Dlup}.

\begin{theorem}\label{Dlup}
Let $Q$ be a $D$-loop. Then

\smallskip
$(1)$ \ all parastrophes of $Q$ coincide with $Q$, or\vspace*{1mm}

$(2)$ \ $Q$ has three classes of isotopic parastrophes: $\{Q,Q_5\}$, $\{Q_1,Q_3\}$, $\{Q_2,Q_4\}$.\vspace*{1mm}

\noindent
The second case holds if and only if $Q\not\sim Q_1$ or $Q\not\approx Q_1$.
\end{theorem}
\begin{proof}
Let $Q$ be a $D$-loop. Then $Q\sim Q$. Thus all its parastrophes are isotopic to $Q$ or they are divided into three classes $\{Q,Q_5\}$, $\{Q_1,Q_3\}, \{Q_2,Q_4\}$ (see Table). By Lemmas \ref{L-26} and \ref{L-28} they are disjoint if and only if $Q\not\sim Q_1$ or $Q\not\approx Q_1$.
\end{proof}

\begin{corollary}\label{C35}
A $D$-loop $Q$ has three classes of isotopic parastrophes if and only if  $Q\not\sim Q_2$ or $Q\not\approx Q_2$.
\end{corollary}

In \cite{Dlup} is proved that parastrophes of a $D$-loop $Q$ are isomorphic to one of the quasigroups $Q$, $Q_1$, $Q_2$. Comparing this fact with our results we obtain

\begin{theorem}
For a $D$-loop $Q$ the following conditions are equivalent:

\smallskip
$(1)$ \ all parastrophes of $Q$ are isomorphic,\vspace*{1mm}

$(2)$ \ $Q$ and $Q_1$ are isomorphic,\vspace*{1mm}

$(3)$ \ $Q$ and $Q_2$ are isomorphic,\vspace*{1mm}

$(4)$ \ $Q_1$ and $Q_2$ are isomorphic.
\end{theorem}

\begin{example}\label{Ex1}
Consider the following three loops.\\

\noindent\hspace*{-3mm}
\begin{tabular}{lcrcccc}
$
\begin{array}{c|ccccccccccc}
\cdot&1&2&3&4&5&6\\ \hline
1&1&2&3&4&5&6\rule{0pt}{10pt}\\
2&2&1&6&5&3&4\\
3&3&6&1&2&4&5\\
4&4&5&2&1&6&3\\
5&5&3&4&6&1&2\\
6&6&4&5&3&2&1
\end{array}
$ 
& \
$
\begin{array}{c|ccccccccccc}
\circ_1&1&2&3&4&5&6\\ \hline
1&1&2&3&4&5&6\rule{0pt}{10pt}\\
2&2&1&5&6&4&3\\
3&3&4&1&5&6&2\\
4&4&3&6&1&2&5\\
5&5&6&2&3&1&4\\
6&6&5&4&2&3&1
\end{array}
$ \
& 
$
\begin{array}{c|ccccccccccc}
\circ_2&1&2&3&4&5&6\\ \hline
1&1&2&3&4&5&6\rule{0pt}{10pt}\\
2&2&1&4&3&6&5\\
3&3&5&1&6&2&4\\
4&4&6&5&1&3&2\\
5&5&4&6&2&1&3\\
6&6&3&2&5&4&1
\end{array}
$
\end{tabular}

\bigskip
The first loop is a D-loop, the second and the third are parastrophes of the first. They are not D-loops and are not isotopic to the first. So this D-loop has three classes of isotopic parastrophes. In this case $Q=Q_5$, $Q_1=Q_3$ and $Q_2=Q_4$.\qed
\end{example}

\section*{\centerline{4. Some consequences}}\setcounter{section}{4}\setcounter{theorem}{0}

Note first of all that the proofs of our results remain true also for the case when $\alpha=\beta=\gamma$.
In this case an anti-isotopism is an anti-isomorphism and an isotopism is an isomorphism. So, the above results will be true if we replace an anti-isotopism by an anti-isomorphism, and an isotopism by an isomorphism. Moreover, an isotopism of parastrophes can be characterized by the identities:
\begin{eqnarray}
\alpha_1(x)\cdot\beta_1(yx)=\gamma_1(y), \label{eq2}\\[2pt]
\beta_2(xy)\cdot\alpha_2(x)=\gamma_2(y), \label{eq3}\\[2pt]
\beta_3(yx)\cdot\alpha_3(x)=\gamma_3(y), \label{eq4}\\[2pt]
\alpha_4(x)\cdot\beta_4(xy)=\gamma_4(y), \label{eq1}\\[2pt]
\beta_5(xy)=\gamma_5(y)\cdot\alpha_5(x), \label{eq5}
\end{eqnarray}
where $\alpha_i,\beta_i,\gamma_i$ are fixed permutations of the set $Q$.

\smallskip
Namely, from our results it follows that 
$$
\begin{array}{lc}
Q \ \ {\rm satisfies} \ \eqref{eq2}\longleftrightarrow Q_1\sim Q\longleftrightarrow Q_3\approx Q,\\[4pt]
Q \ \ {\rm satisfies} \ \eqref{eq3}\longleftrightarrow Q_2\sim Q\longleftrightarrow Q_4\approx Q,\\[4pt]
Q \ \ {\rm satisfies} \ \eqref{eq4}\longleftrightarrow Q_3\sim Q\longleftrightarrow Q_2\approx Q,\\[4pt]
Q \ \ {\rm satisfies} \ \eqref{eq1}\longleftrightarrow Q_4\sim Q\longleftrightarrow Q_1\approx Q,\\[4pt]
Q \ \ {\rm satisfies} \ \eqref{eq5}\longleftrightarrow Q\sim Q\longleftrightarrow Q_5\approx Q.
\end{array}
$$

Lemma \ref{L-23} shows that these identities are universal in some sense, i.e., if one of these identities is satisfied in a quasigroup $Q$, then  in a quasigroup isotopic to $Q$ is satisfied the identity of the same type, i.e., it is satisfied with other permutations.

\medskip
Since $Q\sim Q_1\longleftrightarrow Q\sim Q_2$ we have

\begin{proposition}\label{P41} A quasigroup $Q$ satisfies for some $\alpha_1,\beta_1,\gamma_1\in S_Q$ the identity \eqref{eq2} if and only if for some $\alpha_2,\beta_2,\gamma_2\in Q_S$ it satisfies the identity \eqref{eq3}.
\end{proposition}

As a consequence we obtain the following classification of quasigroups>

\begin{theorem}\label{T42}
Let $Q$ be a quasigroup. Then \vspace{-1mm}
\begin{enumerate}
\item[$\bullet$] \ $Q$ is type $A$ if and only if it satisfies all of the identities $\eqref{eq2}- \eqref{eq5}$,\vspace{-2mm}
\item[$\bullet$] \ $Q$ is type $B$ if and only if it satisfies only \eqref{eq2} and \eqref{eq3},\vspace{-2mm}
\item[$\bullet$] \ $Q$ is type $C$ if and only if it satisfies only \eqref{eq4},\vspace{-2mm}
\item[$\bullet$] \ $Q$ is type $D$ if and only if it satisfies only \eqref{eq1},\vspace{-2mm}
\item[$\bullet$] \ $Q$ is type $E$ if and only if it satisfies only \eqref{eq5},\vspace{-2mm}
\item[$\bullet$] \ $Q$ is type $F$ if and only if it satisfies none of the identities $\eqref{eq2}-\eqref{eq5}$.
\end{enumerate}
\end{theorem}

If all permutations used in $\eqref{eq2}-\eqref{eq5}$ are the identity permutations, then these equations have of the form:
\begin{eqnarray}
x\cdot yx=y, \label{eq22}\\
xy\cdot x=y, \label{eq23}\\
yx\cdot x=y, \label{eq24}\\
x\cdot xy=y, \label{eq21}\\
xy=yx. \label{eq25}
\end{eqnarray}
\noindent
Basing on our results we conclude that 
$$
\begin{array}{lc}
Q \ {\rm satisfies} \ \eqref{eq22}\longleftrightarrow Q=Q_4,\\[4pt]
Q \ {\rm satisfies} \ \eqref{eq23}\longleftrightarrow Q=Q_3,\\[4pt]
Q \ {\rm satisfies} \ \eqref{eq24}\longleftrightarrow Q=Q_2,\\[4pt]
Q \ {\rm satisfies} \ \eqref{eq21}\longleftrightarrow Q=Q_1,\\[4pt]
Q \ {\rm satisfies} \ \eqref{eq25}\longleftrightarrow Q=Q_5.
\end{array}
$$

Since $Q$ satisfies $\eqref{eq23}\longleftrightarrow Q_5=Q_2\longleftrightarrow ((Q_1)_2)_1=Q_2\longleftrightarrow Q_1=((Q_2)_1)_2\longleftrightarrow Q_1=Q_5\longleftrightarrow Q$ satisfies \eqref{eq22}, we see that identities \eqref{eq23} and \eqref{eq22} are equivalent, i.e., $Q$ satisfies \eqref{eq23} if and only if it satisfies \eqref{eq22}.

As a consequence we obtain the stronger version of Theorem 4 in \cite{Lind}.

\begin{theorem} Parastrophes of a quasigroup $Q$ can be characterized by the identities $\eqref{eq22}-\eqref{eq25}$ in the following way: 

\smallskip
$\bullet$ \ $Q=Q_i$ for $1\leqslant i\leqslant 5$ 
if and only if it satisfies all of the identities $\eqref{eq22}-\eqref{eq25}$,

\smallskip
$\bullet$ \ $Q=Q_3=Q_4$, \ $Q_1=Q_2=Q_5$ \ if and only if $Q$ satisfies only \eqref{eq23} and \eqref{eq22},

\smallskip
$\bullet$ \ $Q=Q_2$, \ $Q_1=Q_4$, \ $Q_3=Q_5$ \ if and only if $Q$ satisfies only \eqref{eq24},

\smallskip
$\bullet$ \ $Q=Q_1$, \ $Q_2=Q_3$, \ $Q_4=Q_5$ \ if and only if $Q$ satisfies only \eqref{eq21},

\smallskip
$\bullet$ \ $Q=Q_5$, \ $Q_1=Q_3$, \ $Q_2=Q_4$ \ if and only if $Q$ satisfies only \eqref{eq25},

\smallskip
$\bullet$ \ $Q\ne Q_i\ne Q_j$ for all $1\leqslant i<j\leqslant 5$ 
if and only if $Q$ satisfies none of the  

\hspace*{4mm}identities $\eqref{eq22} -\eqref{eq25}$.
\end{theorem}

\begin{corollary}
Parastrophes of a commutative quasigroup $Q$ coincide with $Q$ or are divided into three classes:  
$\{Q=Q_5\}$, \ $\{Q_1=Q_3\}$, \ $\{Q_2=Q_4\}$.
\end{corollary}

\begin{corollary}
For a commutative quasigroup $Q$ the following conditions are equivalent:

\smallskip
$(1)$ \ all parastrophes of $Q$ coincide with $Q$,\vspace*{1mm}

$(2)$ \ $Q=Q_1$,\vspace*{1mm}

$(3)$ \ $Q=Q_2$,\vspace*{1mm}

$(4)$ \ $Q_1=Q_2$,\vspace*{1mm}

$(5)$ \ $Q$ satisfies at least one of the identities $\eqref{eq22}-\eqref{eq21}$. 
\end{corollary}
\begin{proof}
We prove only the equivalence $(1)\longleftrightarrow (2)$. Other equivalences can be proved in a similar way.

For a commutative $Q$ we have $Q=Q_5$, \ $Q_1=Q_3$, \ $Q_2=Q_4$. If $Q=Q_1$, then
$Q=Q_1=Q_3=Q_5$. Hence $Q_1=Q_5=((Q_2)_1)_2$ which gives $(Q_1)_2=(Q_2)_1$. So, $Q_3=Q_4$, i.e., $(2)$ implies $(1)$. The converse implication is obvious.
\end{proof}
\begin{corollary}
Parastrophes of a boolean group coincide with this group. 
\end{corollary}

Note finally that identities $\eqref{eq22}-\eqref{eq25}$ can be used to determine some autotopisms of quasigroups \cite{D}.

\footnotesize{\rightline{Received \ November 12, 2015}

Faculty of Pure and Applied Mathematics

 Wroclaw University of Technology

Wyb. Wyspia\'nskiego 27

50-370 Wroclaw, Poland

E-mail: wieslaw.dudek@pwr.edu.pl
  }
\end{document}